\documentclass{svproc}

\usepackage[intlimits]{amsmath}
\usepackage{amssymb}
\usepackage{amsfonts}
\numberwithin{equation}{section}
\usepackage{bbm}

\usepackage{xcolor}
\usepackage{tikz}
\usepackage{pgfplots}
\usepackage{float}
\usepackage{subcaption}

\newcommand{\ie}{i.e.\ }
\newcommand{\eg}{e.g.\ }

\newcommand{\wrt}{w.r.t.\ }

\newcommand{\restr}[1]{\raisebox{-.5ex}{$|$}_{#1}}
\newcommand{\norm}[1]{||#1||}
\newcommand{\diff}{\mathop{}\!\mathrm{d}}

\newcommand{\clos}{\operatorname{clos}}

\newcommand{\R}{\mathbb{R}}
\newcommand{\xcal}{\mathcal{X}}
\newcommand{\ycal}{\mathcal{Y}}
\newcommand{\zcal}{\mathcal{Z}}

\usepackage{xargs}
\usepackage[colorinlistoftodos]{todonotes}

\usepackage{url}

\begin{document}
\mainmatter              
\title{An optimally stable approximation of\\ reactive transport using discrete test\\ and infinite trial spaces}
\titlerunning{An optimally stable approximation for reactive transport}
\author{Lukas Renelt\inst{1}, Mario Ohlberger\inst{1} \and
Christian Engwer\inst{1}}
\authorrunning{L. Renelt, M. Ohlberger, C. Engwer} 

\institute{Institute for Analysis and Numerics, University of Münster, Einsteinstr. 62,\\ 48149 M\"unster, Germany.
\url{https://www.uni-muenster.de/AMM} 
\email{\{lukas.renelt,mario.ohlberger,christian.engwer\}@uni-muenster.de}
}

\maketitle              

\begin{abstract}
In this contribution we propose an optimally stable ultraweak Petrov-Galerkin variational formulation and subsequent discretization for stationary reactive transport problems.
The discretization is exclusively based on the choice of discrete approximate test spaces, while the trial space is a priori infinite dimensional.
The solution in the trial space or even only functional evaluations of the solution are obtained in a post-processing step.
We detail the theoretical framework and demonstrate its usage in a numerical experiment that is motivated from modeling of catalytic filters.

\end{abstract}

\section{Introduction}
In this contribution we are considered with stationary reactive transport equations of the form\vspace*{-1em}
\begin{equation}\label{eq:transport:strongProblem}
\begin{cases}
A_\circ u := \nabla\cdot(\vec{b} u) + cu &= f_\circ \qquad  \text{in }\;\Omega,\\
\hfill u &= g_D \hfill \text{on }\Gamma_- \ .\\
\end{cases}
\end{equation}
Here,  $\Omega\subset\R^n$ denotes an open, bounded polyhedral domain with Lipschitz-boundary $\Gamma := \partial\Omega$,  $\vec{b}\in H^1(\mathop{div}, \Omega)$  a
given divergence-free\footnote{We restrict to divergence free velocities. However, non-divergence free fields may be considered as well.} transport field and $c\in L^\infty(\Omega)$ a reaction coefficient.
Moreover, $f_\circ\in L^2(\Omega)$ denotes the source term and $g_D\in L^2(\Gamma_-)$ the  boundary values at the inflow boundary,
where the  in- and outflow boundary parts are defined as 
$
\Gamma_\pm := \{ z\in\Gamma \;|\; \vec{b}(z)\cdot \nu(z) \gtrless 0\}
$
and $\nu(\cdot)$ denotes the outer unit normal. In a weak sense, we interpret $A_\circ: U \to V'$ as an operator that maps from a trial space $U$ to the dual of a test space $V$,
where $U, V$ need to be chosen appropriately.

A known challenge in the numerical treatment of transport problems is the stability of the chosen variational formulation and discretization. 
A common approach for stabilization is the enrichment of the test space by supremizers leading to a Petrov-Galerkin scheme. Depending on the norms on trial and test space, determining the supremizer for a given trial function often requires solving an additional equation with complexity of the full problem.
In this contribution, we use an approach using `optimal' test spaces, \ie spaces that a-priori include all supremizers and thus are unconditionally stable (see \eg\cite{demkowiczDPG1,DahmenHuangSchwab}). However, unlike the discontinuous Petrov-Galerkin method (DPG) and the method defined in~\cite{DahmenHuangSchwab} we discretize by first choosing the discrete test space and
subsequently determine a corresponding trial space as it has been used in~\cite{BrunkenSmetanaUrban,HenningPalitta}.
The resulting approach is computationally more efficient and leads to an optimally stable formulation and discretization.
We show that it is actually not needed to construct a discrete trial space at all since all computations can be performed
using a related normal equation that is defined solely on the test space. Interpreting the solution to the normal equation as a dual variable, this is similar to a (FOS)LL* approach (see~\cite{caiFOSLL} and references therein).
We provide the theoretical framework in Section 2 and the resulting approximation scheme in Section 3.
Numerical experiments for reactive transport in catalytic filters are
given in Section 4. The experiments show 
advantages and challenges 
of the proposed approach. Resulting perspectives for model order reduction based on the proposed framework are discussed in the conclusion.

\section{Ultraweak Petrov-Galerkin variational formulation and related normal equation}
\newcommand{\Ltwoout}{L^2(\Gamma_+, |\vec{b}\vec{\nu}|)}


Following the ideas developed in \cite{BrunkenSmetanaUrban}, we derive an ultraweak variational Petrov-Galerkin formulation for the reactive transport problem~\eqref{eq:transport:strongProblem}
as well as a corresponding normal equation that only depends on the
test space.
We
start with the definition of the formal adjoint operator $A_\circ^*:V \to U'$ given through
\begin{equation}
  v\mapsto A_\circ^*[v], \qquad (A_\circ^*[v])(u) := {(u,-b\nabla v + cv)}_{L^2(\Omega)} + {(u\restr{\Gamma_+}, v\restr{\Gamma_+})}_{\Ltwoout}
\end{equation}
using the weighted inner product
$
{(u,v)}_{L^2(\Gamma_{\pm}, |\vec{b}\vec{\nu}|)} := \int_{\Gamma_\pm} u\,v\,|\vec{b}\vec{\nu}| \diff s .
$

\subsection{Choice of appropriate function spaces}

Let us start with regular spaces $U=C^\infty_{\Gamma_-}(\Omega)$, the space of $C^\infty$-functions vanishing on $\Gamma_-$ and $V=C^\infty(\Omega)$.
For given $v\in V$ the operator $A_\circ^*[v]$ is continuous, and hence in $U'$, if we choose
\begin{equation}
  \norm{u}_U := \sqrt{\norm{u}_{L^2(\Omega)}^2 + \norm{u\restr{\Gamma_+}}_{\Ltwoout}^2}
\end{equation}
as the norm on $U$. Via closure we obtain the ultraweak trial space $\xcal := \clos_{\norm{\cdot}_U}(U)$ equipped with the continuous extension of the norm $\norm{\cdot}_U$.
For this trial space we have the following characterization (without proof).

\begin{proposition}
There exists a linear and continuous trace operator
\begin{equation*}
\gamma: \xcal \rightarrow \Ltwoout
\end{equation*}
as well as a linear and continuous projection operator
\begin{equation}
  \mathrm{pr}_{L^2}: \xcal \rightarrow L^2(\Omega)
\end{equation}
fulfilling $\gamma(u) = u\restr{\Gamma_+}$ and $\mathrm{pr}_{L^2}(u) = u$ for all $u\in U \subset \xcal$.
\end{proposition}

\begin{proposition}\label{thm:trialIsometry}
  The trial space $\xcal$ is isometrically isomorphic to the Sobolev-space $\xcal_{L^2}$ defined as
  \[
  \xcal_{L^2} := L^2(\Omega) \times \Ltwoout
  \]
  equipped with the canonical norm $\norm{(u, \hat{u})}_{\xcal_{L^2}}^2 := \norm{u}_{L^2(\Omega)}^2 + \norm{\hat{u}}_{\Ltwoout}^2$.
  The isometry $\Phi: \xcal \rightarrow \xcal_{L^2}$ is given as
  \begin{equation*}
    \Phi(u) := (pr_{L^2}(u), \gamma(u)).
  \end{equation*} 

\end{proposition}
%

\begin{corollary}
  It holds that
\begin{align*}
  \norm{x}_\xcal^2 &=\quad \norm{\mathrm{pr}_{L^2}(x)}_{L^2(\Omega)}^2 + \norm{\gamma(x)}_{\Ltwoout}^2 = {(x,x)}_\xcal \\
\text{with } \  (x,x')_\xcal &:=\quad {(\mathrm{pr}_{L^2}(x), \mathrm{pr}_{L^2}(x'))}_{L^2(\Omega)} + {(\gamma(x), \gamma(x'))}_{\Ltwoout}.
\end{align*}
In particular, $\xcal$ is a Hilbert space and there exists the Riesz-map $R_\xcal: \xcal' \rightarrow \xcal$.
\end{corollary}
\par
Hence, by continuous extension we can interpret $A_\circ^*[v]$ as an operator acting on $\xcal$, where in the following we will write $(u, \hat{u})\in\xcal$ in the sense of  Prop.~\ref{thm:trialIsometry}, \ie
\begin{equation}
A_\circ^*[v](u,\hat{u}) := {(u, -b\nabla v +cv)}_{L^2(\Omega)} + {(\hat{u}, v\restr{\Gamma_+})}_{\Ltwoout}.
\end{equation}

Provided that $A_\circ^*$ is injective on $V$ (which holds under mild assumptions on the data functions, see~\cite[Prop.~2.2]{BrunkenSmetanaUrban}) we are now in the position to define a norm on $V$ as follows
\begin{equation}
  \norm{v}_V := \norm{A_\circ^*[v]}_{\xcal'} = \norm{R_\xcal(A_\circ^*[v])}_\xcal.
\end{equation}

By a simple variational argument one sees that for a given $v\in V$ the Riesz-representative $r_v := R_\xcal(A_\circ^*[v])\in\xcal$ fulfills $\Phi(r_v) = (-b\nabla v + cv, v\restr{\Gamma_+})$. We thus obtain
\begin{equation*}
\norm{v}_V^2 = \norm{r_v}_\xcal^2 = \norm{\Phi(r_v)}_{\xcal_{L^2}}^2 = \norm{-b\nabla v + cv}_{L^2(\Omega)}^2 + \norm{v\restr{\Gamma_+}}_{\Ltwoout}^2.
\end{equation*}
By closure we finally obtain the test space $\ycal := \clos_{\norm{\cdot}_V}(V)$ of our ultraweak formulation equipped with the norm
\begin{equation}
\norm{y}_\ycal := \norm{A^*[y]}_{\xcal'}
\end{equation}
where $A^*:\ycal \rightarrow \xcal'$ is the continuous extension of $A_\circ^*$ to $\ycal$.

The following proposition gives a characterization of the test space $\ycal$:
\begin{proposition}
  The test space $\ycal$ is isomorphic to the Sobolev-space
  \begin{equation}
H^1(\vec{b}, \Omega) := \{v\in L^2(\Omega) \;|\; \vec{b}\nabla v \in L^2(\Omega)\}.
  \end{equation}
\end{proposition}

\subsection{Optimally stable ultraweak formulation and normal equation}
With the definitions of the trial space $\xcal$, the test space $\ycal$ and the adjoint operator $A^*$ we are now prepared to give an ultraweak 
variational formulation for reactive transport as follows.
\begin{definition}[Ultraweak variational formulation of reactive transport]
$u\in\xcal$ is called a solution of the ultraweak variational formulation, if it satisfies
\begin{equation} \label{eq:continuousProblem}
 \quad {(u,A^*[v])}_{\xcal \times \xcal'} = f(v) \qquad \forall v\in\ycal.
\end{equation}
with right-hand side
$
f(v) := {(f_\circ, v)}_{L^2(\Omega)} - {(g_D, v)}_{\Ltwoout}.
$
\end{definition}
\begin{proposition}\label{prop:genericIsometry}
The mappings $A: \xcal\rightarrow\ycal'$ and $A^*:\ycal\rightarrow \xcal'$ are isometries, \ie
\begin{equation*}
  \ycal = A^{-*}\xcal', \quad \xcal = A^{-1}\ycal'\vspace*{-1em}
\end{equation*}
and\vspace*{-1em}
\begin{equation*}
  \norm{A}_{\mathcal{L}(\xcal,\ycal')} = \norm{A^*}_{\mathcal{L}(\ycal,\xcal')}
  = \norm{A^{-1}}_{\mathcal{L}(\ycal',\xcal)} = \norm{A^{-*}}_{\mathcal{L}(\xcal',\ycal)}
  = 1.
\end{equation*}
\end{proposition}
\begin{proof}
  The proof follows the argumentation in~\cite[Prop.~2.1]{DahmenHuangSchwab}.
\end{proof}

As a consequence we obtain the following corollary, which shows optimal stability of the ultraweak variational formulation.
\begin{corollary}[Optimal stability]\label{cor:generalProblem}
The variational formulation~\eqref{eq:continuousProblem}
is well-posed and has optimal condition number $\kappa_{\xcal,\ycal}(A) := \norm{A}\,\norm{A^{-1}} = 1$.
\end{corollary}

Since by Proposition\ \ref{prop:genericIsometry} every $u\in\xcal$ has a representation $u = R_\xcal A^*[w]$ for some $w\in\ycal$ we can substitute $u$
in the ultraweak formulation~\eqref{eq:continuousProblem} to obtain the equivalent (continuous) normal equation.
\begin{definition}[Normal equation of the ultraweak formulation]
$w\in\ycal$ is called a solution of the normal equation of the ultraweak formulation, if it satisfies
\begin{equation}\label{eq:transport:continuousNormalEq}
	 {(A^*[w], A^*[v])}_{\xcal'} \;=\;  f(v) \qquad \forall v\in\ycal
\end{equation}
or equivalently
\begin{equation}
 {(-b\nabla w + cw, -b\nabla v + cv)}_{L^2(\Omega)} + {(\gamma(w),\gamma(v))}_{\Ltwoout} \quad =\quad f(v).
\end{equation}
\end{definition}
This is essentially a LL*-method (see \eg~\cite{caiFOSLL}) applied to the minimization of the residual energy $\norm{A[u]-f}_{\ycal'}$.
\begin{remark}\label{rmk:strongTestProblem}
Let $w$ denote a solution of the normal equation~\eqref{eq:transport:continuousNormalEq}. If $w$ is regular enough (\eg $w\in C^2(\Omega)$), then $w$ solves the degenerated Poisson problem
\begin{equation}
\begin{cases}
  -\nabla\cdot(D\nabla w) &= f_\circ \qquad \text{in}\; \Omega \\
  -b\nabla w &= g_D \qquad \text{on}\; \Gamma_- \\
  (D\nabla w)\vec{\nu} + w &= 0 \qquad \text{on}\; \Gamma_+ \\
\end{cases}
\end{equation}
with a rank-$1$ diffusion tensor $D := \vec{b}\otimes\vec{b}$.
\end{remark}
The equivalent formulation of the normal equation will serve as the starting point for the definition of an optimal stable approximation method in the following section.

\section{A test-space only discretization}

As indicated in the previous section, we propose to use the normal equation~\eqref{eq:transport:continuousNormalEq} to define an optimally stable approximation scheme.
It is thus obvious, that a discretization can be fully based on a discrete approximate test space $\ycal^\delta$.

\subsection{Discrete normal equation and functional reconstruction}

Let $\ycal^\delta \subseteq \ycal$ be a conforming discretization of the optimal test space (\ie using a standard Lagrange finite element space). Based on~\eqref{eq:transport:continuousNormalEq} we then define the \textit{discrete} normal equation using Galerkin-projection.
\begin{equation}\label{eq:discreteNormalEq}
\text{Find}\; w^\delta \in \ycal^\delta: \quad {(A^*[w^\delta], A^*[v^\delta])}_{L^2(\Omega)} = f(v^\delta) \qquad \forall v^\delta \in \ycal^\delta.
\end{equation}
Note that this is still an optimally conditioned problem. Given the discrete solution $w^\delta$ we may reconstruct the discrete solution $u^\delta = A^*[w^\delta]$. Technically, this solution lies in the finite-dimensional subspace $\xcal^\delta := A^*[\ycal^\delta] \subseteq \xcal$, however, due to its non-accessible structure this space is of no practical use.

Previous work often used knowledge of the structure of $A^*$ to determine a larger, more traditional (DG-)space $\zcal^\delta \supsetneq \xcal^\delta$ and then assembled the matrix $\underline{A}$ representing the operator $A^*: \xcal^\delta \rightarrow \zcal^\delta$ in the respective standard FE-bases. In this case, one can determine the system matrix $\underline{A}^{NE}$ of the normal equation as $\underline{A}^{NE} = \underline{A}^T \underline{M}_\zcal \underline{A}$ (where $\underline{M}_{\zcal}$ denotes the inner-product matrix in $\zcal^\delta$), solve the linear system\vspace*{-0.5em}
\begin{equation}\label{eq:linearEquationSystem}
  \underline{A}^{NE}\underline{w} = \underline{f}\vspace*{-0.5em}
\end{equation}
and compute the coefficients $\underline{u}$ of $u^\delta \in\xcal^\delta \subset \zcal^\delta$ in the basis of $\zcal^\delta$ by simply computing $\underline{u} = \underline{A}\,\underline{w}$.

However, this is suboptimal as the construction of a discrete larger space $\zcal^\delta$ is only feasible or even possible with suitable additional assumptions on the data, e.g. (elementwise) constant data functions. For non-constant reaction or velocities one has to resort to a nonconforming choice $\zcal^\delta \not\supset \xcal^\delta$ introducing an additional projection error which might be difficult to estimate or control.

Here, we propose an approach that avoids ever computing a matrix $\underline{A}$ representing the operator $A^*$. The system matrix of the normal equation $\underline{A}^{NE}$ can also be directly assembled in a basis of $\ycal^\delta$ which means basically assembling a normal equation using the full infinite dimensional trial space $\xcal$. The reconstruction $u^\delta := A^*[w^\delta]$ is now seen as an element of $\xcal$ (we technically know that it lies in the finite dimensional subspace $\xcal^\delta \subset \xcal$ but this does not give us any useful information). The crucial insight is that in almost all applications only functional evaluations of $u^\delta$ are needed. Examples include point-evaluations for the visualization of $u^\delta$ or the computation of quantities of interest via numerical quadrature (\ie $\norm{u^\delta}$). Therefore, we replace the reconstruction by functional evaluations and e.g. do a pointwise reconstruction. Note that in this way we do not introduce any additional projection error.

\subsection{Conditioning of the system matrix and solving the linear system}
Solving the linear equation system~\eqref{eq:linearEquationSystem} is actually quite a challenging task - a problem that has to our knowledge not been discussed so far. Although Problem~\eqref{eq:discreteNormalEq} is optimally stable in theory, the condition of the system matrix $\underline{A}^{NE}$ still scales quadratically in the inverse grid width $h^{-1}$ and is thus a significant challenge even for moderately large problems. To better understand these seemingly conflicting statements consider the non-symmetric formulation of~\eqref{eq:discreteNormalEq}:
\begin{equation}\label{eq:discreteNonsymmetricEq}
\text{Find}\; u^\delta \in \xcal^\delta: \quad (u^\delta, A^*[v^\delta]) = f(v^\delta) \qquad \forall v^\delta \in \ycal^\delta.
\end{equation}
Let $\{\psi_i\}_{i=1}^{N}$ be a basis of $\ycal^\delta$ (\eg a finite element basis). Then, the set $\{ \varphi_i \}_{i=1}^N$, $\varphi_i := A^*[\psi_i]$ forms a basis of $\xcal^\delta$ and the matrix $\underline{A}$ representing $A^*$ in these bases is the identity matrix. The condition of the system matrix $\underline{A}^{NE}$ is still of order $\mathcal{O}(h^{-2})$ since the trial functions $\varphi_i$ have, contrary to classic finite elements, in this case a magnitude of $\mathcal{O}(h^{-1})$.

As mentioned in Remark~\ref{rmk:strongTestProblem}, the normal equation can also be seen as the weak form of a specific Poisson-problem with rank-deficient diffusion tensor $D$. In the following numerical experiments we thus employed an algebraic multigrid for preconditioning and a conjugate gradient (CG) solver - methods that are known to perform well for this type of problems. For more complex problems (\eg for velocity fields with (locally) small magnitude) the efficient preconditioning and solving of the linear equation system~\eqref{eq:linearEquationSystem} still needs further investigation.

\section{Numerical experiments}

\begin{figure}[t]
  \hfill
  \begin{subfigure}{0.3\textwidth}
    \centering
    \begin{tikzpicture}[scale=0.8\textwidth/1cm]
      \draw[black,thick] (0.0,0.0) rectangle (1.0,1.0);
      \draw[red,ultra thick] (0.0, 1.0) -- (0.0, 2./3);
      \node[black] at (0.1, 5/6) {$\Gamma_{in}$};
      \draw[cyan,ultra thick] (1.0, 0.0) -- (1.0, 1./3);
      \node[black] at (0.85, 1/6) {$\Gamma_{out}$};
      \node[black] at (0.5, 0.07) {$\Gamma_0$};
      \node[black] at (0.
      5, 0.9) {$\Gamma_0$};
      \filldraw[fill=black!40!white, draw=black, opacity=0.8] (0.0, 0.4) rectangle (1.0, 0.6);
      \node[black] at (0.5,0.5) {$\Omega_{reac}$};
    \end{tikzpicture}
    \vspace*{0.5em}
    \caption{Problem setup}
  \end{subfigure}
  \hfill
  \begin{subfigure}{0.3\textwidth}
    \centering
    \includegraphics[width=0.9\textwidth]{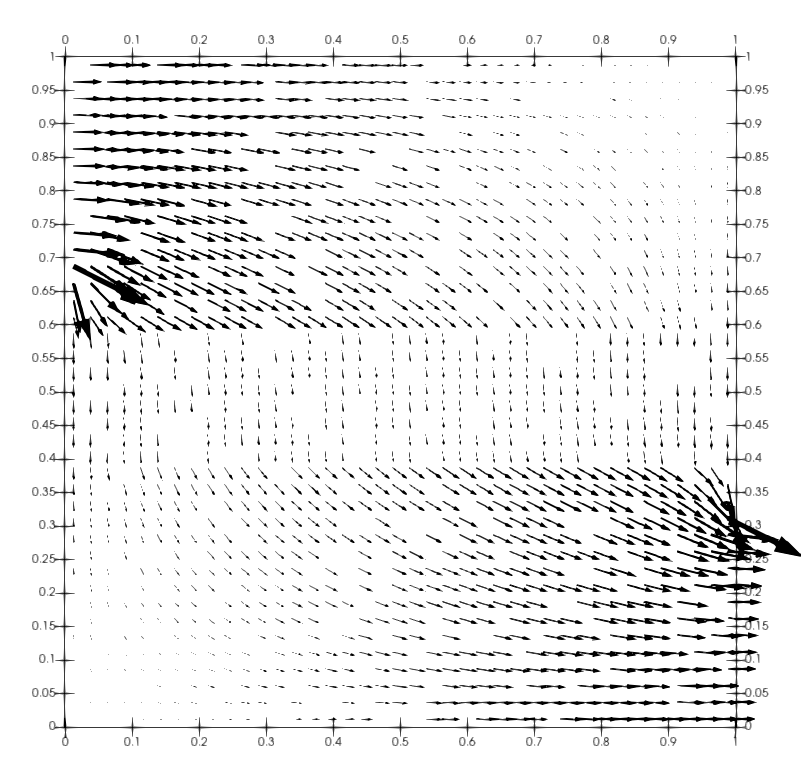}
    \caption{Velocity field $\vec{b}$}
  \end{subfigure}
  \hfill
  \begin{subfigure}{0.3\textwidth}
    \centering
    \includegraphics[width=0.9\textwidth]{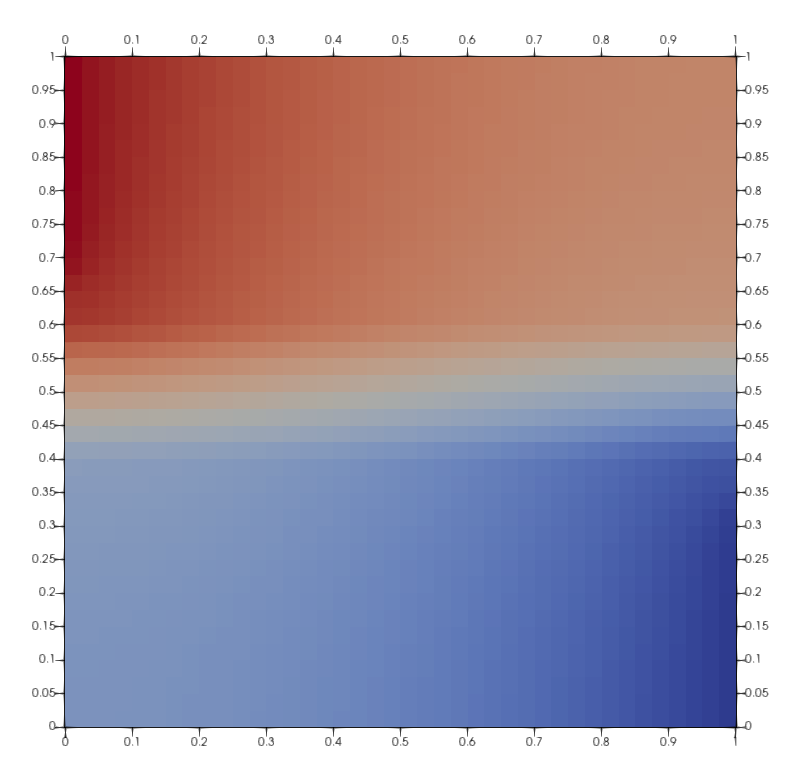}
    \caption{Pressure field $p$}
  \end{subfigure}\vspace*{-.2cm}
  \caption{\label{fig:testcases:catalysator} Problem setup and Darcy velocity field $\vec{b}$}\vspace*{-.75em}
\end{figure}
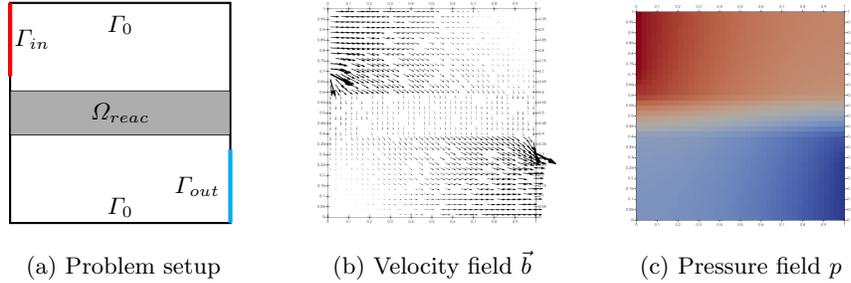

The discretization scheme and the following experiments were implemented using the DUNE-framework~\cite{bastian2021dune} and the DUNE-PDELab discretization toolbox%
\footnote{\url{https://www.dune-project.org/modules/dune-pdelab}}
\footnote{\url{https://doi.org/10.5281/zenodo.7950882}}.
We consider the reactive transport of a pollutant inside a catalytic
filter. Let $\Omega$ be the unit square with boundary $\Gamma :=
\partial\Omega$. We assume that the velocity $\vec{b}$ is given
as the solution to the Darcy-equation\vspace*{-0.5em}
\begin{equation*}
    \nabla\cdot \vec{b} = 0, \qquad  
    \vec{b} = -k\nabla p \\
\end{equation*}
subject to the boundary conditions\vspace*{-0.5em}
\begin{equation*}
  p = 1 \quad\text{on}\; \Gamma_{in}, \qquad
  p = 0 \quad\text{on}\; \Gamma_{out}, \qquad
  \vec{b} = 0 \quad\text{on}\; \Gamma_0 \;:= \Gamma \setminus (\Gamma_{in} \cup \Gamma_{out}).
\end{equation*}
Here, $k \in L^\infty(\Omega)$ denotes the permeability field. We introduce the reactive domain (washcoat) $\Omega_{reac} \subset \Omega$ and assume that the permeability in the reactive domain is significantly smaller, \ie $k = k_{min}\cdot\mathbbm{1}_{\Omega_{reac}} + \mathbbm{1}_{\Omega \backslash \Omega_{reac}}$, $k_{min}\in\R^+$. Similarly, the reaction function is given by $c = c_0 \cdot\mathbbm{1}_{\Omega_{reac}}$, $c_0\in\R^+$. All chosen parameters of the problem are summarized in Table~\ref{tab:params}.

\begin{table}[t]
  \begin{center}
    \renewcommand{\arraystretch}{1.4}
    \begin{tabular}{|c|c|c|c|c|c|c|c|}
    \hline
    $\Omega$ & $\Omega_{reac}$ & $\Gamma_{in}$ & $\Gamma_{out}$ & $g_D(z)$ & $f_\circ(x)$ & $k_{min}$ & $c_0$ \\
    \hline
    $[0,1]^2$ & $[0,1]\times [0.4,0.6]$ & $\{0\}\times (\tfrac{2}{3},1)$ & $\{1\} \times (0,\tfrac{1}{3})$ &$\sin(3\pi z)^2$ & $0$ & $10^{-1}$ & $0.5$ \\
    \hline
    \end{tabular}
  \end{center}
  \caption{Chosen parameters for the catalytic filter problem}\vspace*{-1cm}
  \label{tab:params}
\end{table}

In order to inspect the reconstruction $u^\delta$ we perform a voxel-wise evaluation on a refined mesh (\ie a projection into $\mathbb{P}^0$). In Fig.~\ref{fig:solutions} this is depicted alongside a solution obtained by a first order SIPG-method. In particular one notices nonaligned gradients when using a first order discretization of $\ycal$ (Fig.~\ref{fig:firstOrder}). However, since $u^\delta$ is only a $L^2$-like best approximation we do not expect $u^\delta$ to provide meaningful information \wrt the gradient (although the magnitude of the derivation still needs to converge in order to achieve $L^2$-convergence). Additionally, $u^\delta$ is the $\xcal_{L^2}$-best approximation from a non-standard space $\xcal^\delta$ where no information about its approximation qualities are given.

\begin{figure}[ht]
  \begin{subfigure}{0.315\textwidth}
    \centering
    \includegraphics[width=\linewidth]{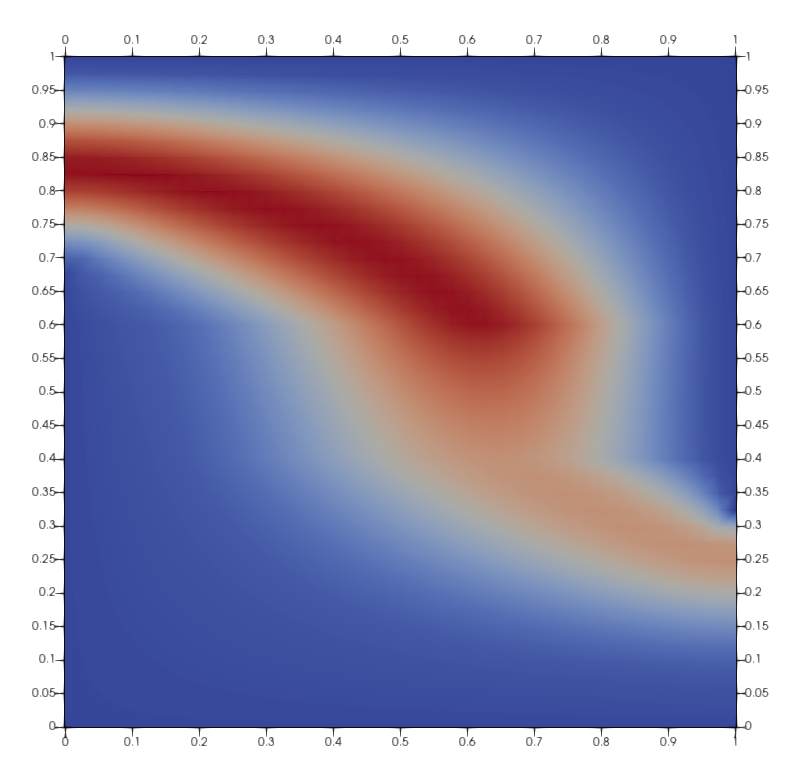}
    \caption{Solution computed with a first-order SIPG-scheme.}
  \end{subfigure}
  \hfill
  \begin{subfigure}{0.32\textwidth}
    \centering
    \includegraphics[width=.985\linewidth]{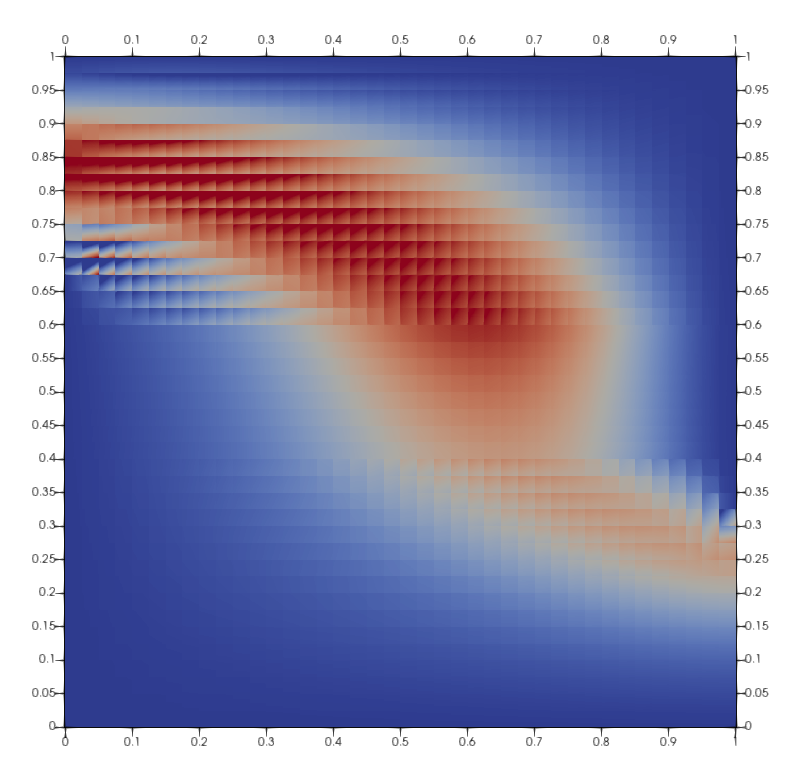}
    \caption{\label{fig:firstOrder}Reconstructed solution for first order test functions.}
  \end{subfigure}
  \hfill
  \begin{subfigure}{0.345\textwidth}
    \centering
    \includegraphics[width=.92\linewidth]{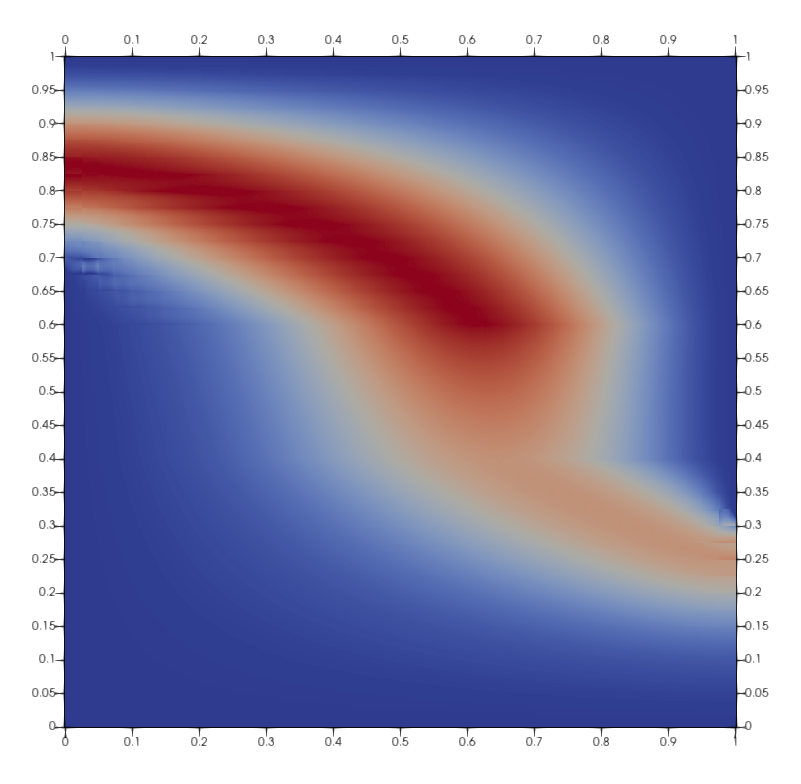}
    \caption{\label{fig:secondOrder}Reconstructed solution for second order test functions}
  \end{subfigure}\vspace*{-.2em}
  \caption{\label{fig:solutions} Solutions obtained by different discretization methods. All of them are based on a structured grid with $h^{-1}=40$.}\vspace*{-.5em}
\end{figure}

Finally, we investigate the convergence of the method under $h$-refinement. In~\cite{keithApriori} it was shown that the rates depend solely on the regularity of the test space solution. In~\cite{BrunkenSmetanaUrban} the dependence of the convergence rate on the regularity of the inflow condition (which in turn determines the regularity of the test space solution) has also been numerically evaluated for a linear transport problem. In our chosen catalytic filter problem we observe a convergence order of about $1.2$ for linear and $2.3$ for quadratic test functions (Fig.~\ref{fig:hconvergence}).

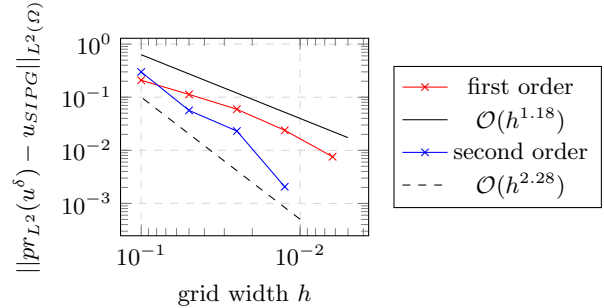
\begin{figure}[ht]
\centering
\begin{tikzpicture}
  \begin{loglogaxis}[
          width=0.4\linewidth, 
          grid=major, 
          grid style={dashed,gray!30}, 
          xlabel=Global basis size, 
          xlabel=grid width $h$,
          x dir=reverse,
          ylabel= $\norm{pr_{L^2}(u^\delta) - u_{SIPG}}_{L^2(\Omega)}$,
          legend style={at={(1.1,0.5)},anchor=west},
        ]

        \addplot[color=red , mark=x]
        table[x=gridwidth,y=l2error,col sep=comma] {conv_data_first_order.csv};
        \addlegendentry{first order};
        \addplot[color=black, domain=5e-3:0.1]{10*x^(1.2)};
        \addlegendentry{$\mathcal{O}(h^{1.18})$};
        \addplot[color=blue , mark=x]
        table[x=gridwidth,y=l2error,col sep=comma] {conv_data_second_order.csv};
        \addlegendentry{second order};
        \addplot[color=black, dashed, domain=1e-2:0.1]{20*x^(2.3)};
        \addlegendentry{$\mathcal{O}(h^{2.28})$};
  \end{loglogaxis}
\end{tikzpicture}
\vspace{-1em}
\caption{\label{fig:hconvergence} Convergence under $h$-refinement}\vspace{-2em}
\end{figure}

\section{Conclusion}\vspace*{-0.75em}
In this contribution we derived an ultraweak, optimally stable
formulation for reactive
transport.
In contrast to previous work we
did not introduce additional boundary conditions on the test space but
instead imposed them weakly by including boundary terms in the adjoint
operator. We also showed that a standard FE-discretization of the
test space is sufficient to solve the normal equations and
perform classic functional evaluations of the reconstructed solution
without discretizing the trial space explicitly.
Future
work will concern the robust solving of the normal equations,
as well as using model order reduction techniques to efficiently solve
parameter-dependent reactive transport problems.  \vspace*{-1.25em}

\subsubsection{Acknowledgements} 
The authors acknowledge funding by the BMBF under contract 05M20PMA
and by the Deutsche Forschungsgemeinschaft under Germany’s Excellence
Strategy EXC 2044 390685587, Mathematics M\"unster: Dynamics --
Geometry -- Structure.\vspace*{-1.25em}

\bibliographystyle{abbrv}
\bibliography{ms}

\end{document}